\documentclass[
final,
]{amsart}

\usepackage[english]{babel}
\usepackage[utf8]{inputenc}

\usepackage[T1]{fontenc}
\usepackage{lmodern}
\usepackage{amssymb}

\usepackage{enumitem}
\usepackage[notref,notcite,color]{showkeys}
\usepackage{verbatim}
\usepackage{url}

\newcommand*{\Theorem}{Theorem}
\newcommand*{\Proposition}{Proposition}
\newcommand*{\Lemma}{Lemma}
\newcommand*{\Corollary}{Corollary}
\newcommand*{\Definition}{Definition}
\newcommand*{\Remark}{Remark}
\newcommand*{\Notation}{Notation}

\theoremstyle{plain}
\newtheorem{theorem}{\Theorem}

\newtheorem{corollary}[theorem]{\Corollary}
\newtheorem{lemma}[theorem]{\Lemma}
\theoremstyle{definition}
\newtheorem{definition}[theorem]{\Definition}
\theoremstyle{remark}
\newtheorem{remark}[theorem]{\Remark}

\usepackage{prettyref}
\newrefformat{thm}{\Theorem~\ref{#1}}
\newrefformat{pro}{\Proposition~\ref{#1}}
\newrefformat{cor}{\Corollary~\ref{#1}}
\newrefformat{lem}{\Lemma~\ref{#1}}
\newrefformat{rem}{\Remark~\ref{#1}}
\newrefformat{def}{\Definition~\ref{#1}}

\renewcommand{\phi}{\varphi}
\renewcommand{\epsilon}{\varepsilon}

\usepackage{ifthen,xspace}
\newcommand{\lpfont}{\mathsf}

\newcommand*{\lp}[2][]{\ensuremath{
    \lpfont{
      \ifthenelse{\equal{#1}{}}{#2}{#1 \text{-} #2}}}\xspace}
\newcommand*{\lpp}[3][]{\ensuremath{
    \lpfont{
      \ifthenelse{\equal{#1}{}}{#2}{#1 \text{-} #2}}(#3)}\xspace}
\newcommand*{\lph}[3][]{\ensuremath{
    \lpfont{
      \ifthenelse{\equal{#1}{}}{{#2}^{#3}}{#1 \text{-} {#2}^{#3}}}}\xspace
}

\newcommand*{\Weak}{{\upharpoonright}}

\newcommand*{\lf}[1]{\lpfont{#1}}

\newcommand*{\ls}[2][]{\ensuremath{
  \lpfont{
      \ifthenelse{\equal{#1}{}}{#2}{#1 \text{-} #2}}}\xspace}

\newcommand*{\WIDEomega}[1]{\vphantom{\lpfont{\widehat{#1}}}\smash[t]{\lpfont{\widehat{#1}}}^\omega}

\newcommand{\IMPL}[1][]{
  \ifthenelse{\equal{#1}{}}{\mathop{\rightarrow}}{\mathop{\stackrel{#1}{\longrightarrow}}}}

\newcommand*{\OR}{\mathrel{\lor}}
\newcommand*{\NOT}{\neg}
\newcommand*{\IFF}[1][]{
  \ifthenelse{\equal{#1}{}}{\mathrel{\leftrightarrow}}{\mathrel{\stackrel{#1}{\longleftrightarrow}}}}

\newcommand*{\Quantor}[2]{{#1 #2}\,}
\newcommand*{\Forall}[1]{\Quantor{\forall}{#1}}
\newcommand*{\Exists}[1]{\Quantor{\exists}{#1}}

\newcommand*{\Nat}{\ensuremath{\mathbb{N}}}
\newcommand*{\Real}{\ensuremath{\mathbb{R}}}
\newcommand*{\Rat}{\ensuremath{\mathbb{Q}}}

\title{On the strength of weak compactness}
\author{Alexander P.\ Kreuzer}
\address{Fachbereich Mathematik, Technische Universit\"{a}t Darmstadt\\
Schlossgartenstra{\ss}e~7, 64289 Darmstadt, Germany
}
\email{akreuzer@mathematik.tu-darmstadt.de}
\thanks{The author is supported by the German Science Foundation (DFG Project KO 1737/5-1).}
\thanks{I am grateful to Ulrich Kohlenbach for useful
discussions and suggestions for improving the presentation
of the material in this article.}
\subjclass[2010]{Primary 03F60; Secondary 03D80, 03B30}
\keywords{Bolzano-Weierstra{\ss} principle, weak sequential compactness, Turing degree, abstract Hilbert space}

\begin{document}
\begin{abstract}
  We study the logical and computational strength of weak compactness in the separable Hilbert space~$\ell_2$.

  Let \lp[weak]{BW} be the statement the every bounded sequence in $\ell_2$ has a weak cluster point.
  It is known that \lp[weak]{BW} is equivalent to \ls{ACA_0} over \ls{RCA_0} and thus that it is equivalent to (nested uses of) the usual Bolzano-Weierstra{\ss} principle \lp{BW}.

  We show that \lp[weak]{BW} is instance-wise equivalent to the \lp[\Pi^0_2]{CA}. This means that for each $\Pi^0_2$ sentence $\lf{A}(n)$ there is a sequence $(x_i)_{i\in\Nat}$ in $\ell_2$, such that one can define the comprehension functions for $\lf{A}(n)$ recursively in a cluster point of $(x_i)_i$.
  As consequence we obtain that the degrees $d\ge_T 0''$  are exactly those degrees that contain a weak cluster point of any computable, bounded sequence in $\ell_2$.
  Since a cluster point of any sequence in the unit interval $[0,1]$ can be computed in a degree low over $0'$ (see \cite{aK}), this show also that instances of \lp[weak]{BW} are strictly stronger than instances of \lp{BW}.

  We also comment on the strength of \lp[weak]{BW} in the context of abstract Hilbert spaces in the sense of Kohlenbach and show that his construction of a solution for the functional interpretation of weak compactness is optimal, cf.~\cite{uKgf}.  
\end{abstract}

\maketitle

We investigate the computational and logical strength of weak sequential compactness in the separable Hilbert space~$\ell_2$.

The strength of weak compactness has so far only been studied in the context of proof mining, see \cite{uK10,uKgf}. There general Hilbert spaces in a more general logical system are considered.  It is straightforward to deduce from this analysis that weak compactness for $\ell_2$ is equivalent to \ls{ACA_0} over \ls{RCA_0}.

In this paper we refine this result and show that  weak compactness on $\ell_2$ is instance-wise equivalent to \lp[\Pi^0_2]{CA} over \ls{RCA_0}.
This means that for each bounded sequence in $\ell_2$ one can uniformly compute a function $f$ such that from a comprehension function for $\Forall{x}\Exists{y} f(x,y,n)=0$ one can compute a weak cluster point and vice versa.

 As consequence we obtain that the degrees $d \ge_T 0''$ are exactly those degrees that compute a weak cluster point for each computable bounded sequence in $\ell_2$ and that there is a computable bounded sequence in $\ell_2$ such that from a cluster point of this sequence one can compute $0''$.

This shows that instances of the Bolzano-Weierstra{\ss} principle for weak compactness are strictly stronger than instances of the usual Bolzano-Weierstra{\ss} principle.

This paper is organized as follows: first the Hilbert space~$\ell_2$ is defined. This definition follows \cite{sS99,AS06}. Then the actual results are proven (Theorems~\ref{thm:weakbw} and \ref{thm:rev}) and we show that the result can also be formulated for abstract Hilbert spaces, in the sense of Kohlenbach \cite{uK08} (\prettyref{thm:weakbwabs}). 
As corollary of this we obtain that Kohlenbach's analysis of the weak compactness functional $\Omega^*$ in \cite{uKgf} is optimal (\prettyref{cor:omega}).
At the end,  we reformulate the result of the analysis in terms of the Weihrauch lattice (\prettyref{rem:weihrauch}).

\medskip

\begin{definition}[vector space, {\cite[II.10]{sS99}}]
  A \emph{countable vector space $A$ over a countable field $K$} consists of a set $\lvert A\rvert \subseteq \Nat$
  with operations $+\colon \lvert A\rvert \times \lvert A\rvert \to \lvert A \rvert$ and
  $\cdot \colon \lvert K \rvert \times \lvert A\rvert \to \lvert A \rvert$ and
  a distinguished element $0\in \lvert A \rvert $
  such that $(\lvert A \rvert, +,\cdot, 0)$ satisfies the usual axioms for a vector space over $K$.
\end{definition}

\begin{definition}[Hilbert space, {\cite[Definition~9.3]{AS06}}]\label{def:hs}
  A \emph{(real) separable Hilbert space $H$} consists of a countable vector space $A$ over $\Rat$ together 
  with a function $\langle \cdot, \cdot \rangle \colon A\times A \to \Real$ satisfying
  \begin{enumerate}
  \item $\langle x, x \rangle \ge 0$,
  \item $\langle x, y \rangle = \langle y, x \rangle$,
  \item $\langle ax + by, z \rangle = a\langle x, z \rangle + b \langle y, z \rangle$,
  \end{enumerate}
  for all $x,y,z\in A$ and $a,b\in \Rat$.
\end{definition}
The inner product on $H$ induces a pseudonorm $\lVert x \rVert := \sqrt{\langle x,x\rangle}$.
We think of the Hilbert space $H$ as the completion of $A$ under the pseudometric $d(x,y)=\lVert x - y \rVert$. Thus an element of $H$ consists of a sequence $(x_n)_{n\in\Nat}\subseteq A$, such that $d(x_n,x_m)< 2^{-n}$ for all $m>n$.
The inner product $\langle \cdot , \cdot \rangle$ is continuously extended to the whole space~$H$.

A Hilbert space is finite dimensional if it is spanned by finitely many vectors. If this is not the case we say that it is infinite dimensional.

Avigad, Simic showed in \cite[Theorem~10.9]{AS06} that every Hilbert space~$H$ in the sense of Definition~\ref{def:hs} has a orthonormal basis. Since each such Hilbert space is separable this basis is at most countable.

As consequence of this each two infinite dimensional (separable) Hilbert spaces are isomorphic over \ls{RCA_0}, see \cite[Corollary~10.11]{AS06}. Thus we many restrict our attention to $\ell_2$, as given by the following definition.

\begin{definition}[$\ell_2$, {\cite[II.10.2]{sS99}}]
  Let $A=(\lvert A \rvert, +, \cdot, 0)$ be a vector space over $\Rat$, where
  $\lvert A\rvert$ is the set of all finite sequences of rational numbers $\langle r_0,\dots,r_m\rangle$, such that either $m=0$ or $r_m\neq 0$. Addition is defined by putting $\langle r_0,\dots,r_m\rangle + \langle s_0,\dots,s_n\rangle = \langle r_0+s_0,\dots,r_k+s_k\rangle$ where $r_i,s_i=0$ for $i>m,n$ and $k=\max\{i\mid i=0 \OR r_i+s_i\neq 0\}$. For scalar multiplication put $q\cdot \langle r_0,\dots,r_m\rangle =\langle 0\rangle$ if $q=0$ and $\langle q\cdot r_0,\dots,q\cdot r_m\rangle$ otherwise.

  The space $\ell_2$ is defined to be the Hilbert space consisting of $A$ with the inner product
  \[
  \big\langle \langle r_0,\dots,r_m\rangle, \langle s_0,\dots,s_n \rangle\big\rangle = \sum_{i=0}^{\max(n,m)} r_i s_i
  .\]
\end{definition}

The canonical orthonormal basis $(e_n)_n$ of $\ell_2$ is given by 
\[
e_n = \langle \underbrace{0,\dots,0}_{\text{$n$ times}} , 1\rangle.
\]

\begin{definition}[projection]
  Let $M$ be a closed linear subspace of a Hilbert space $H$. A point $y\in M$ is called \emph{projection} of $x\in H$ if $x-y$ is orthogonal to (each element of) $M$.

  A bounded linear operator $P_M$ on $H$ that maps each point of $H$ to its projection on $M$ is called \emph{projection function for $M$}.
\end{definition}
Usually projections are defined differently, see e.g.\ \cite[Definition~12.1]{AS06}. Avigad, Simic showed that this definition is over \ls{RCA_0} equivalent to the usual definition, see \cite[Lemma~12.2]{AS06}.

We immediately obtain the following lemma:
\begin{lemma}\label{lem:proj}
  Let $N\subset \Nat$ and $M$ be the subspace of $\ell_2$ that is spanned by $\{e_n \mid n\in N\}$.
  Then \ls{RCA_0} proves that the projection $P_M$ of $\ell_2$ onto the space $M$ exists.
\end{lemma}
\begin{proof}
  The projection of an element $\langle r_0,\dots, r_m\rangle$ of the space $\lvert A \rvert$ is given by
  $\langle r_0',\dots, r_{m'}'\rangle$, where $r'_i = r_i$ if $n\in N$ and $r'_i = 0$ if $n\notin N$ and $m' = \max\{ i \le m \mid r_i \neq 0 \OR i=0 \}$.

  It is easy to show that $P_M$ is linear and that is bounded by $1$ (at least on $\lvert A \rvert$). From this one can deduce that $P_M$ is continuous and continuously extend it to the full space $\ell_2$.
\end{proof}

\begin{definition}[weak convergence]
  We say that a sequence $(x_i)_{i\in\Nat}$ of elements of a Hilbert space $H$ \emph{converges weakly} to a point $x$ if 
  \begin{equation}\label{eq:weak}
    \Forall{y\in H} \lim_{i\to\infty} \langle y, x_i \rangle = \langle y, x \rangle
    .
  \end{equation}

  The \emph{Bolzano-Weierstra{\ss} principle for weak convergence} is defined to be the statement that for every bounded sequence $(x_i)_{i\in\Nat}$ of elements of $H$ there exists a point $x$ such that a subsequence of $(x_i)_i$ converges weakly to $x$.
  This principle is abbreviated by \lp[weak]{BW}.
  The restriction of this principle to a fixed sequence $(x_i)_{i\in\Nat}$ is denoted by \lpp[weak]{BW}{(x_i)_i}.

  If $H$ has an orthonormal basis it is sufficient to have \eqref{eq:weak} only for all $y$ in the basis. 
\end{definition}

\begin{lemma}\label{lem:pweak}
  Projections are weakly continuous in the sense that if $x$ is the weak limit point of a sequence $(x_i)_{i\in \Nat}$, then $Px$ is the weak limit point of $(Px_i)_{i\in\Nat}$ for any projection $P$.
\end{lemma}
\begin{proof}
  Follows from the definition of the projection and the continuity of $\langle \cdot,\cdot \rangle$.
\end{proof}

An instance of $\Pi^0_2$-comprehension given by a $\Pi^0_2$ formula $\lf{A}(n)$ is the statement 
\begin{align*}
  \Exists{g} \Forall{n} & \left(g(n)=0 \IFF \lf{A}(n)\right).
  \intertext{Since all $\Pi^0_2$-formulas $\lf{A}(n)$ can be written as $\Forall{x}\Exists{y} txyn = 0$ for a (primitive recursive) term $t$, we can rewrite this as}
  \Exists{g} \Forall{n} & \left(g(n)=0 \IFF \Forall{x}\Exists{y} txyn = 0 \right).
\end{align*}
This statement will be abbreviated by \lpp[\Pi^0_2]{CA}{t}. (In order to be able to formulate this in \ls{RCA_0} we will use the conservative extension by all primitive recursive functions of it.)

\begin{theorem}\label{thm:weakbw}
  For each instance $\lf{A}(n)\equiv [\Forall{x}\Exists{y} txyn=0]$ of \lp[\Pi^0_2]{CA} there exists a bounded sequence $(x_i)_{i\in\Nat}$ in $\ell_2$, such that
  \[
  \ls{RCA_0} \vdash \lpp[weak]{BW}{(x_i)_{i\in\Nat}} \IMPL \lpp[\Pi^0_2]{CA}{t}
  .
  \]

  Moreover, the sequence $(x_i)_{i\in\Nat}$ can be primitive recursively and uniformly computed from $t$, i.e.\ there is a primitive recursive functional $F$ such that $x_i=F(t,i)$.
\end{theorem}
\begin{proof}
  Define 
  \[
    f(n,i):= \max \{ x \le i \mid \Forall{x' < x} \Exists{y<i} (t(x',y,n)=0) \}
  .\]
  It is clear that $\lambda i . f(n,i)$ is increasing for each $n$.

  \noindent\textit{Claim 1.}
  \[
  \lf{A}(n) \quad \text{if{f}} \quad \text{$\lambda i. f(n,i)$ is unbounded, i.e.\  $\Forall{k}\Exists{i} (f(n,i)>k)$}
  .\]
  \noindent\textit{Proof of Claim 1.}
  \begin{itemize}
  \item The right to left direction follows immediately from the definition of $f$.
  \item For the left to right direction fix an $n$. We will show that not the right side implies not the left side.

    Hence assume that $\lambda i . f(n,i)$ is bounded by $k$, i.e.
    \begin{equation}\label{eq:bounded}
      \Forall{i} (f(n,i) \le k).
    \end{equation}
    By $\Sigma^0_1$-induction we may assume that $k$ is minimal and thus 
    \[
    \Exists{i} (f(n,i)=k)
    .\]
    From the definition of $f$ we obtain
    \[
    \Forall{x < k} \Exists{y} (t(x,y,n)=0)
    .\]
    Together with \eqref{eq:bounded} we obtain that
    \[
    \Forall{y} (t(k,y,n)\neq 0)
    \]
    and hence $\NOT \lf{A}(n)$.
  \end{itemize}
  This proofs the claim.

  Let 
  \[
  y_{n,i}:= e_{\langle n, f(n,i) \rangle}
  .\]
  The sequence $(y_{n,i})_{i\in\Nat}$ is obviously bounded by $1$ and hence possesses for each $n$ a weak cluster point $y_n$.

  \noindent\textit{Claim 2.}
  \begin{itemize}
  \item $\lVert y_n \rVert =_\Real 0$, if $\lf{A}(n)$ and 
  \item $\lVert y_n \rVert =_\Real 1$, if $\NOT \lf{A}(n)$.
  \end{itemize}
  
  \noindent\textit{Proof of Claim 2.}
  \begin{itemize}
  \item If $\lf{A}(n)$ is true, then $\lambda i . f(n,i)$ is unbounded and hence $\langle e_j, y_{n,i}\rangle$ eventually becomes $0$. Therefore $y_{n,i}$ converges weakly to $0$.
  \item If $\lf{A}(n)$ is false, then $\lambda i . f(n,i)$ is bounded. By $\Sigma^0_1$-induction we obtain a smallest upper bound $k$ and since $\lambda i . f(n,i)$ is increasing we obtain that
  $\lim_{i\to\infty} f(n,i) = k$.
  As consequence we obtain that $y_{n,i}$ eventually becomes constant $e_{\langle n,k\rangle}$ and hence that $y_n=e_{\langle n,k\rangle}$ and $\lVert y_n \rVert =_\Real 1$.
  \end{itemize}
  This proves the claim.

  We parallelize this process to obtain the comprehension function for $\lf{A}(n)$. For this let
  \[
  x_i := \sum_{n=0}^{i} 2^{-\frac{n+1}{2}}y_{n,i}
  .\]
  Since the $y_{n,i}$ are orthogonal for different $n$, we obtain by Pythagoras that
  \[
  {\lVert x_i \rVert}^2 = \sum_{n=0}^{i} 2^{-(n+1)}  {\lVert y_{n,i} \rVert}^2 \le 1
  \]
  and thus that $(x_i)$ is bounded.

  By \lpp[weak]{BW}{(x_i)_i} there exists a weak cluster point $x$ of $(x_i)$.
  Let now $M_n$ be the closed linear space spanned by $\{  e_{\langle n , k \rangle} \mid k\in\Nat \}$. 
  By definition the subspaces $M_n$ are disjoint (except for the $0$ vector) for different $n$, and $y_{n,i}\in M_n$ for all $i,n$.

  By \prettyref{lem:proj} the projections $P_{M_n}$ onto the spaces $M_n$ exist. For this projections we have
  \begin{align*}
    P_{M_n}(x_i) &= 2^{-\frac{n+1}{2}} y_{n,i} \quad \text{for $n \ge i$}
    .
    \intertext{Since $P_{M_n}$ is weakly continuous, see \prettyref{lem:pweak}, we get}
    P_{M_n} (x) &= 2^{-\frac{n+1}{2}} y_n
    .
  \end{align*}  
  Now Claim~2 yields that $\lVert P_{M_n}(x) \rVert =_\Real 0$ if $\lf{A}(n)$ and $\lVert P_{M_n}(x) \rVert =_\Real 2^{-\frac{n+1}{2}}$ if $\NOT\lf{A}(n)$.
  Hence the function
  \[
  g(n):=
  \begin{cases}
    0 & \text{if $\lVert P_{M_n}(x) \rVert(n+1) <_\Rat 2^{-\frac{n+1}{2}}$,} \\
    1 & \text{otherwise,}
  \end{cases}
  \]
  where $\lVert P_{M_n}(x) \rVert(n+1)$ is a $2^{-(n+1)}$ good rational approximation of $\lVert P_{M_n}(x) \rVert$, provides a comprehension function and solves the theorem.

  It is clear that $(x_i)$ is primitive recursive in $t$.
\end{proof}

As immediate consequence we obtain the following corollary:
\begin{corollary}
  There is a sequence $(x_i)_i$ of elements in $\ell_2$ such that from a cluster point $x$ of this sequence one can compute any element of the second Turing jump $0''$.
\end{corollary}
\begin{proof}
  Take for $\lf{A}(n)$ in \prettyref{thm:weakbw} the $\Pi^0_2$ statement that the Turing machine $\{n\}^{0'}(n)$ halts.
\end{proof}

Kohlenbach studies weak compactness in the context of arbitrary abstract Hilbert spaces, see  \cite{uK08,uK10}. 
By abstract Hilbert space we mean that the Hilbert space is added as a new type to the system together with the Hilbert space axioms and that the space is not coded as sequences of numbers. 
With this one can analyze Hilbert spaces without referring to a concrete space like $\ell_2$ and one does not automatically obtain a separable Hilbert space but can analyze general Hilbert spaces.

We do not introduce the notation for abstract Hilbert spaces here but refer the reader to \cite[Chapter~17]{uK08}.
We show now that the statement of \prettyref{thm:weakbw} is also applicable in this context:
\begin{theorem}\label{thm:weakbwabs}
  Let $\ls{\WIDEomega{PA}\Weak\mathnormal{[X,\langle\cdot,\cdot\rangle]}}$ be the extension of $\ls{\WIDEomega{PA}\Weak}$ by the abstract Hilbert space $X$ with the scalar product $\langle \cdot,\cdot\rangle$ and let now $\lp[weak]{BW}$ denote the Bolzano-Weierstra{\ss} principle for weak compactness in $X$.
  Then
  \begin{multline*}
  \ls{\WIDEomega{PA}\Weak\mathnormal{[X,\langle\cdot,\cdot\rangle]}} + \lp[\Pi^0_1]{CP} \vdash 
  \left(
    \Exists{(e_i)_{i\in\Nat}} \Forall{i,j} \langle e_i, e_j \rangle = \delta_{ij} \right) \\
  \IMPL 
  \Exists{(x_i)_{i\in\Nat}} \left(\lpp[weak]{BW}{(x_i)_i} \IMPL \lpp[\Pi^0_2]{CA}{t}\right)
  .
  \end{multline*}
  In other words, if $X$ is provably infinite dimensional, then \prettyref{thm:weakbw} also holds with $\ell_2$ replaced by $X$.
\end{theorem}
\begin{proof}
  The only step in the proof of \prettyref{thm:weakbw} that does not formalize in \ls{\WIDEomega{PA}\Weak\mathrm{[X,\langle\cdot,\cdot\rangle]}} is projection of $x$ onto $M_n$, i.e.\ \prettyref{lem:proj}, since this depends on the coding of $\ell_2$.

  We show now how to obtain this projection of $x$ in this system. For this consider
  \begin{align*}
    {\lVert x \rVert}^2 & = \langle x , x \rangle = \lim_{i\to\infty} \langle x ,x_i \rangle \\
    & = \lim_{i\to\infty} \sum_{n=0}^i 2^{-(n+1)} \langle x, y_{n,i} \rangle \\
    & \le  \lim_{i\to\infty} \sum_{n=0}^k 2^{-(n+1)} \langle x, y_{n,i} \rangle + 2^{-k} \qquad\text{for each $k$} \\
    & = \sum_{n=0}^k   2^{-(n+1)} \lim_{i\to\infty} \langle x, y_{n,i} \rangle + 2^{-k} 
    .
  \end{align*}
  Now 
  \begin{equation}\label{eq:xyscal} 
  \langle x, y_{n,i} \rangle = \lim_{j\to\infty} \langle x_j,y_{n,i}\rangle = 2^{-(n+1)} \lim_{j\to\infty} \langle y_{n,j},y_{n,i}\rangle
  .\end{equation}
  Thus, by the definition of $y_{n,i}$ the term $\langle x, y_{n,i} \rangle$ is monotone in $i$ and in particular for each $n$ there is an $m$, such that
  \[
  \lim_{i\to\infty} \langle x, y_{n,i}\rangle= \langle x, y_{n,i'}\rangle \quad\text{for $i'\ge m$}
  .\]
  By \lp[\Pi^0_1]{CP} there is now an $m$ which does it for all $n\le k$.
  Hence, we obtain
  \begin{align*}
  \Forall{k}\Exists{i}\, {\lVert x \rVert}^2  & \le \sum_{n=0}^k   2^{-(n+1)}  \langle x, y_{n,i} \rangle + 2^{-k}.
  \intertext{By \eqref{eq:xyscal} the term  $\langle x, y_{n,i} \rangle$ is either $0$ or $2^{-(n+1)}$, hence}
  & = \sum_{n=0}^k   {\langle x, y_{n,i} \rangle}^2 + 2^{-k}.
  \end{align*}
  Thus, $\sum_{n=0}^k \langle x, y_{n,i} \rangle y_{n,i}$ is a $2^{-k/2}$ good approximation of $x$ consisting of finite linear combinations of $(e_i)$. Using an application of \lp[QF]{AC} one easily obtains a sequence of approximations converging to $x$ at the rate $2^{-k}$. Using this one can obtain $P_{M_n}(x)$ like in \prettyref{lem:proj}.

 This proves the theorem.
\end{proof}
By applying the functional interpretation to this we obtain the following corollary:
\begin{corollary}\label{cor:omega}
  Let $\Omega$ be a solution of the functional interpretation of \lp[weak]{BW} then for every $n \ge 1$ there  are terms in $T_n$, such that the application of $\Omega$ to these terms is (extensionally) equal to a function definable in the $T_{n+2}$ but not in $T_{n+1}$.
\end{corollary}
\begin{proof}
  Let $\lf{A}$ be the statement that the function $f_{\omega_{n+1}}$ from the fast growing hierarchy is total. It is well known that the statement $\lf{A}$ cannot be proven in \lp[\Sigma^0_{\mathnormal{n}+2}]{IA} but can be proven using a suitable instance of \lp[\Sigma^0_{\mathnormal{n}+3}]{IA}, see \cite[II.3.(d)]{HP98}. Thus a solution of the functional interpretation of $\lf{A}$ cannot be found in $T_{n+1}$ but can be found in $T_{n+2}$.

  Let $\ls{\WIDEomega{PA}\Weak\mathnormal{[X,\langle\cdot,\cdot\rangle, (e_i)_{i\in\Nat}]}}$ be the extension of $\ls{\WIDEomega{PA}\Weak\mathnormal{[X,\langle\cdot,\cdot\rangle]}}$ by the constant $(e_i)_i$, which can be majorized by $\lambda i . 1$, and the axiom $\Forall{i,j\in\Nat} \langle e_i,e_j \rangle =_\Real \delta_{ij}$.
  For this system the metatheorem \cite[Theorem~17.69.2)]{uK08}, see also \cite{GK08},
  \begin{itemize}
  \item relativized to the fragment \ls{\WIDEomega{PA}\Weak} of $\mathcal{A}^\omega$, cf.~\cite[Section~17.1, p.~382]{uK08}  and
  \item extended by the constant $(e_i)_i$ and the purely universal axiom for it, cf.~\cite[Section~17.5]{uK08}
  \end{itemize}
  holds.

  By \prettyref{thm:weakbwabs} a suitable instance of \lp[weak]{BW} can reduce an instance of \lp[\Sigma^0_{\mathnormal{n}+3}]{IA} to \lp[\Sigma^0_{\mathnormal{n}+1}]{IA}. Thus the system $\ls{\WIDEomega{PA}\Weak\mathnormal{[X,\langle\cdot,\cdot\rangle, (e_i)_{i\in\Nat}]}} + \lp[\Sigma^0_{\mathnormal{n}+1}]{IA}$ proves that a suitable instance of \lp[weak]{BW} implies $\lf{A}$.
  Applying the metatheorem to this statement yields terms in $T_{n}$ such that an application of these terms to $\Omega$ yields a solution of the functional interpretation of $\lf{A}$. 

  This prove the corollary.
\end{proof}
This shows that Kohlenbach's analysis of $\Omega^*$ (a majorant of a solution of the functional interpretation of \lp[weak]{BW}) in \cite{uKgf} is optimal.

This analysis and actually even his proof of weak compactness for abstract Hilbert spaces  \cite[Theorem~11]{uK10} shows that
only two nested instances of \lp[\Pi^0_1]{CA} (plus some uses of \lp{WKL}) are needed to proof an instance of \lp[weak]{BW}.
Thus, the lower bound on the strength of instances of \lp[weak]{BW} from the Theorems \ref{thm:weakbw} and \ref{thm:weakbwabs} is strict in the senses that there is no instance of \lp[\Pi^0_3]{CA} which is implied by an instance of \lp[weak]{BW}. 

We now give a reversal for the special case of $\ell_2$ and analyze the exact computational content:
\begin{theorem}\label{thm:rev}
  For each bounded sequence $(x_i)_{i\in\Nat}$ in $\ell_2$ one can compute uniformly and primitive recursively an function $t$ such that
  \[
  \ls{RCA_0} \vdash \lpp[\Pi^0_2]{CA}{t} \IMPL \lpp[weak]{BW}{(x_i)_{i\in\Nat}}
  .
  \]

  In particular, each bounded and computable sequences of $\ell_2$ has a weak cluster point computable in $0''$.
\end{theorem}
\begin{proof}
  We show that provably in \ls{RCA_0} a cluster point of $(x_i)_i$ can be computed in the second Turing jump. The result follows then from the fact that any function computable in the second Turing jump is recursive in a suitable instance of \lp[\Pi^0_2]{CA}.

  We assume that $(x_i)_i$ is bounded by $1$. 

  Note that the Bolzano-Weierstra{\ss} theorem for the space $[-1,1]^\Nat$ (with the product metric $d\big((x_i)_i,(y_i)_i\big) = \sum_{i=0}^{\infty} \frac{\min(\lvert x_i-y_i\rvert, 1)}{2^{i+1}}$) is instance-wise equivalent to the Bolzano-Weierstra{\ss} theorem for $[-1,1]$. This can easily be seen from the fact that the Bolzano-Weierstra{\ss} theorem for $[-1,1]$ is instance-wise equivalent to the theorem for the Cantor space $2^\Nat$ and the fact that $2^\Nat$ is isomorphic to ${(2^\Nat)}^\Nat$.

  Hence by \cite{aK} one can find a cluster point of the sequence 
  \[
  y_i := \big(\langle e_0,x_i\rangle, \langle e_1,x_i \rangle, \dots \big)
  \]
  in $[-1,1]^\Nat$ by computing an infinite path trough a $\Sigma^0_1$-tree. Call this cluster point $c=(c_0,c_1,\dots)\in [-1,1]^\Nat$.
  
  \noindent
  \textit{Claim.} $\sum_{j=0}^\infty c_j \le 1$
  
  \noindent
  \textit{Proof of claim.} Since the elements of $y_i$ are elements of a Hilbert space and are norm bounded by $1$ we have that $\sum_{j=0}^k (y_i)_j^2\le 1$.
  Now for each $k$ and for each $\epsilon$ there is an $y_i$ such that $\lvert c_j - (y_i)_j\rvert \le \epsilon$ for $j\le k$ and hence 
  \[
  \sum_{j=0}^k (c_j)^2 \le \sum_{j=0}^k ((y_i)_j+\epsilon)^2 \le 1 + 3(k+1) \epsilon
  .\]
  From this follows the claim.

  Now one easily checks that the sequence $(z_i)_{i\in\Nat}$ with $z_i := \langle c_0,\dots,c_i\rangle$ converges in the $\ell_2$\nobreakdash-norm to a weak cluster point $x$ of $(x_i)_i$. This convergence is monotone in the sense that $\lVert z_i \rVert \le \lVert z_{i+1} \rVert$ thus the limit point $x$ can be computed in the Turing jump of $(z_i)_i$. 

  The point $x$ is provably computable in the second Turing jump of $(x_i)_i$ because $c$ is by the low basis theorem (\cite{JS72}) computable in a degree provably low over the first Turing jump. (The proof of the low basis theorem is effective and formalizes in \ls{RCA_0}.)
  Therefore the jump of $(z_i)_i$ and thus $x$ is computable in the second Turing jump.
\end{proof}

With this we can classify the computational strength of weak compactness on $\ell_2$:
\begin{corollary}
  For a Turing degree $d$ the following are equivalent:
  \begin{itemize}
  \item $d \ge_T 0''$ and
  \item $d$ computes a weak cluster point for each computable, bounded sequence in~$\ell_2$.
  \end{itemize}
\end{corollary}
As consequence we obtain that the Bolzano-Weierstra{\ss} principle for weak compactness is instance-wise strictly stronger than the Bolzano-Weierstra{\ss} principle for the unit interval $[0,1]$, cf.~\cite{aK}.

\begin{remark}[Weihrauch lattice]\label{rem:weihrauch}
  The proofs of the Theorems~\ref{thm:weakbw} and \ref{thm:rev} can also be used to classify the Bolzano-Weierstra{\ss} principle for weak compactness in $\ell_2$ in the Weihrauch lattice.
  We do not introduce the notation for the Weihrauch lattice but refer the reader to \cite{BGM}.

  Let $\mathsf{BWT}_{\text{weak-}\ell_2} :\subseteq (\ell_2)^\Nat \rightrightarrows \ell_2$ be the partial multifunction which maps bounded sequences of $\ell_2$ to a weak cluster point of that sequence.

  The proof of Theorem~\ref{thm:weakbw} immediately yields that 
  \[
  \mathsf{BWT}_{\text{weak-}\ell_2} \ge_{\mathrm{W}} \widehat{\mathsf{LPO}}\circ \widehat{\mathrm{\mathsf{LPO}}} \equiv_{\mathrm{W}} \mathrm{lim}^{(2)}
  .\]

  Whereas the proof of \prettyref{thm:weakbw} yields that
  \[
  \mathsf{BWT}_{\text{weak-}\ell_2} \le_{\mathrm{W}} \mathsf{MCT} \ast \mathsf{BWT}_{\Real^\Nat}
  .\]
  The function $\mathsf{BWT}_{\Real^\Nat}$ is used to compute the cluster point $c\in \Real^\Nat$, the function $\mathsf{MCT}$ is used for the convergence of $(\lVert z_i\rVert)_i$.
 By the same argument as in the proof $\mathsf{BWT}_\Real \equiv_{\mathrm{W}} \mathsf{BWT}_{\Real^\Nat}$. 
 Since all of these multifunctions are cylinders one may also strengthen the reducibility to strong Weihrauch reducibility.
 Thus
  \begin{align*}
    \mathsf{BWT}_{\text{weak-}\ell_2} & \le_{\mathrm{sW}} \mathsf{MCT} \ast_{\mathrm{s}} \mathsf{BWT}_{\Real} \\
    & \le_{\mathrm{sW}} \mathrm{lim} \ast_{\mathrm{s}} \mathfrak{L'} \\
    & \le_{\mathrm{sW}} \mathrm{lim} \ast_{\mathrm{s}} \mathfrak{L}_{1,2} \\
    & \equiv_{\mathrm{sW}} \mathrm{lim} \circ \mathrm{lim}
    .
  \end{align*}
  (For the last equivalence see \cite[Corollary~8.8]{BGM}, which is a consequence of an analysis of the low basis theorem in the Weihrauch lattice, see \cite{BBP}.)

  In total we obtain that
  \[
  \mathsf{BWT}_{\text{weak-}\ell_2} \equiv_{\mathrm{sW}} \mathrm{lim}^{(2)}
  .\]
  As consequence we also obtain that $\mathsf{BWT}_{\text{weak-}\ell_2} >_{\mathrm{sW}} \mathsf{BWT}_{\Real}$.
\end{remark}

\bibliographystyle{amsplain}
\bibliography{weak}

\end{document}